\documentclass[11pt, oneside]{article}

\usepackage{amsmath, amssymb, amsthm, dsfont, mathabx}
\usepackage{geometry}
\usepackage{epsfig}
\usepackage{xcolor}
\usepackage{hyperref}
\usepackage[font=small,labelfont=bf, width=.85\textwidth]{caption}

\newtheorem{theorem}{Theorem}[section]

\newtheorem{proposition}[theorem]{Proposition}
\newtheorem{lemma}[theorem]{Lemma}

\newtheorem{conjecture}[theorem]{Conjecture}
\newtheorem{observation}[theorem]{Observation}

\newtheorem{computation}{Computation}[section]

\begin{document}

\title{List-edge-coloring triangulations with maximum degree at most five}
\author{
Joshua Harrelson\footnote{Department of Mathematics and Statistics, Middle Georgia State University, Macon, GA 31206 USA; joshua.harrelson@mga.edu}
\qquad
Jessica McDonald\footnote{Department of Mathematics and Statistics,
Auburn University,
Auburn, AL 36849 USA; mcdonald@auburn.edu; Supported in part by Simons Foundation grant \#845698.}
}

\date{}

\maketitle

\begin{abstract}  
We prove that triangulations with maximum degree at most 5 satisfy the List-Edge-Coloring Conjecture.
\end{abstract}

\section{Introduction}
An \emph{edge list assignment} for a graph $G$ is a function $L$ that assigns to each edge $e \in E(G)$ a list of colors $L(e)$. Given such an $L$, an \emph{$L$-edge-coloring} of $G$ is a (proper) edge-coloring of $G$ such that every edge $e$ is given a color from $L(e)$. Note that the classical notion of a $k$-edge-coloring of $G$ can be viewed as an $L$-edge-coloring for the list assignment $L$ defined by $L(e) = \{1,\ldots, k\}$ for all $e \in E(G)$. We say a graph $G$ is \emph{$k$-list-edge-colorable} if it is $L$-edge-colorable for every edge list assignment $L$ such that $|L(e)| \geq k$ for all $e \in E(G)$. The \emph{list-chromatic index} of $G$, denoted $\chi'_\ell(G)$, is the minimum $k$ such that $G$ has a $k$-list-edge coloring. We immediately get that $\chi'_\ell(G) \geq \chi'(G) \geq \Delta$  for every graph $G$, where $\chi'(G)$ is the chromatic index of $G$ (the minimum $k$ such that $G$ is $k$-edge-colorable), and $\Delta:=\Delta(G)$ is the maximum degree of $G$. 

In this paper we consider every graph to be simple, hence Vizing's Theorem \cite{Viz} says that $\chi'(G)\leq \Delta+1$ for all graphs $G$. Vizing \cite{Vz} conjectured that this upper bound also holds for list-edge colorings.

\begin{conjecture}[Vizing \cite{Vz}]
\label{vizcon}
If $G$ is a graph, then $\chi'_\ell(G) \leq \Delta+1$.
\end{conjecture}

Conjecture \ref{vizcon} has been verified for all graphs with $\Delta \leq 4$. The $\Delta=3$ case was proved by Vizing \cite{Vz} in 1976 and independently by Erd\H{o}s, Rubin, and Taylor \cite{ERT} in 1979. The $\Delta=4$ case of Conjecture \ref{vizcon} was proved in 1998 by Juvan, Mohar, \v{S}krekovski \cite{JMS}. Since there are graphs with $\Delta=3$ and $\Delta=4$ having $\chi'(G)=\Delta+1$, these results are tight. However, in general, we may hope for more than Conjecture \ref{vizcon}. The famous List-Edge Coloring Conjecture (LECC), which follows, has been attributed to many sources, some as early as 1975 (see e.g.~\cite{JT}).

\begin{conjecture}[LECC]
\label{lecc}
If $G$ is a graph, then $\chi'(G)=\chi'_\ell(G)$.
\end{conjecture}

The LECC is true for a number of special families, most notably bipartite graphs due to Galvin \cite{G} in 1995 (see also \cite{MP} for an extension). As an example of how far away this conjecture still is however, consider that it has not yet been established for all even cliques (odd cliques were established by H\"{a}ggkvist and Janssen \cite{HJ}, and cliques of order equal to a prime plus one were established by Schauz \cite{s3}).

More is known about Conjectures \ref{vizcon} and \ref{lecc} for planar graphs, where both edge-coloring and list-edge-coloring are somewhat simpler. While  it is NP-complete to decide whether a graph has chromatic index $\Delta$ or $\Delta+1$ (Holyer \cite{Ho}), all planar graphs with $\Delta\geq 7$ have $\chi'(G)=\Delta$ (Sanders and Zhao \cite{SZ}, Zhang \cite{Zh}). We may therefore expect $\chi'_\ell(G)=\Delta$ for all planar graphs $G$ with $\Delta\geq 7$, but this has only been established for $\Delta \geq 12$ (by Borodin, Kostochka and Woodall \cite{BKW}). Conjecture \ref{vizcon} has been pushed further, and we now know it holds for planar graphs when $\Delta \geq 8$ (Bonamy \cite{Bo}). This leaves the planar case of Conjecture \ref{vizcon} open for graphs with $5 \leq \Delta \leq 7$. In this paper we show that the stronger LECC holds for $\Delta\leq 5$ when $G$ is a triangulation (i.e.~when all faces are triangles).

\begin{theorem}\label{main} If $G$ is a triangulation with $\Delta \leq 5$, then $\chi'_l(G)=\chi'(G).$
\end{theorem}

Most of the above-mentioned papers utilize the discharging method in their proofs. A particularly insightful proof by Cohen and Havet \cite{CH} highlights the trouble that triangular faces cause in such arguments (their proof shows Conjecture \ref{vizcon} for $\Delta \geq 9$, which was previously established by Borodin \cite{B}). Theorem \ref{main}
deals with these troublesome cases for $\Delta \leq 5$. In fact, there are not so many such cases: in Section 2 of this paper we will show that there are exactly twelve triangulations with $\Delta\leq 5$, eight of which have $\Delta=5$. In Section 3 we note that some of these graphs are known to satisfy the LECC and we explain how the recent work of Dvo\v{r}\'{a}k \cite{Dv} can be used to computationally show the remaining graphs also satisfy the List-Edge Coloring Conjecture.

\section{Twelve triangulations with \texorpdfstring{$\Delta\leq 5$}{leq5} }

Triangulations with $\Delta\leq5$ are special in that only finitely many degree sequences are possible, via Euler's formula. For what follows we let $V_i=\{v\in V(G):\deg(v)=i\}$. 

\begin{lemma}
\label{count}
If $G$ is a triangulation with $\Delta\leq 5$, then either $G=K_3$ or $V(G)=V_3\cup V_4\cup V_5$ and $12= 3|V_3|+2|V_4|+|V_5|$.
\end{lemma}
\begin{proof}
The only triangulation with a vertex of degree 2 or less is the triangle $K_3$ so, we may assume that $V(G)=V_3\cup V_4\cup V_5$. Since $G$ is a triangulation, Euler's formula gives us $|E(G)|=3|V(G)|-6$. Applying the degree-sum formula this yields the following:
$$12=6|V(G)|-\sum_{v\in V(G)}\deg(v)$$
Substituting $|V(G)|=|V_3|+|V_4|+|V_5|$ and $\sum_{v\in V(G)} \deg(v)=3|V_3|+4|V_4|+5|V_5|$ gives our desired result.
\end{proof}

Note that if we apply the argument of Lemma \ref{count} to triangulations with vertices of degree $6$, then the $|V_6|$ term will cancel, meaning that $|V_6|$ would be unrestricted.

\begin{table}
\begin{tabular}{c|ccccccccccccccccccc}
 Case  & 1& 2& 3& 4& 5& 6& 7&8&9&10&11&12&13&14&15&16&17&18&19\\
 \hline
 $|V_3|$ & 4&3&3&2&2&2&2&1&1&1&1&1&0&0&0&0&0&0&0\\
 $|V_4|$ & 0&0&1&0&1&2&3&0&1&2&3&4&0&1&2&3&4&5&6\\
 $|V_5|$ & 0&3&1&6&4&2&0&9&7&5&3&1&12&10&8&6&4&2&0 \\
\end{tabular}
\caption{degree sequences of $|V_3|$, $|V_4|$, $|V_5|$ satisfying $12= 3|V_3|+2|V_4|+|V_5|$.}
\label{cases}
\end{table}

There are only 19 linear combinations of $|V_3|$, $|V_4|$, and $|V_5|$ that satisfy the equation in Lemma \ref{count}, and they are listed in Table \ref{cases}. We will hereafter refer to each of these potential degree sequences by their case number, as listed in this table. Note that cases 8 and 14 are impossible since they each contain a set of 5 vertices of degree 5 all of which must be adjacent to one another, forcing a copy of $K_5$. In the remainder of this section, we will provide arguments to eliminate all but 11 of the remaining cases, and then show that each of these cases corresponds to a unique triangulation, with case $i$ corresponding to graph $G_i$ in Figure \ref{AllTriangulations}. We define $\mathcal{G}$ as the set of these 11 graphs together with $K_3$. Specifically, given the pictures in Figure \ref{AllTriangulations}, we define $\mathcal{G}=\{G_i:i\in \{0,1,4,6,7,11,13,15,16,17,18,19\}\}$. 

\begin{theorem}\label{Pics} $\mathcal{G}$ is the set of all triangulations with $\Delta(G)\leq5$. 
\end{theorem}

\begin{figure}[hb]
    \centering
    \includegraphics[height=8.5cm]{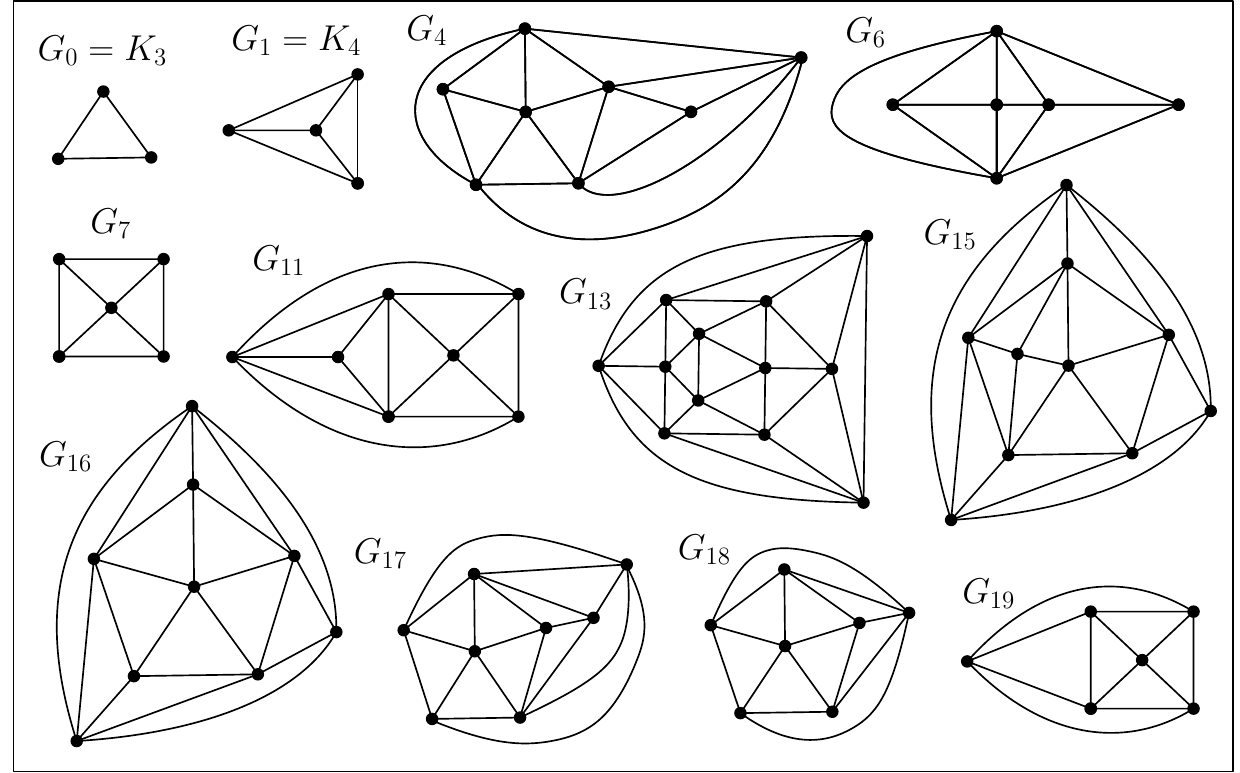}
    \caption{{The set $\mathcal{G}$ of all triangulations with $\Delta\leq 5$.}}
    \label{AllTriangulations}
  \end{figure}

It is important to note that there is an alternative to the argument we are about to present for Theorem \ref{Pics}: it can also be derived by computer, using the results of Brinkmann and McKay on \emph{plantri} (see \cite{BM}).

We start our work towards Theorem \ref{Pics} with the following observation and lemma. 

\begin{observation}
\label{samef}
\emph{Let $G$ be a triangulation with path $P=\{x,y,z\}$ for $x,y,z \in V(G)$. If $xy$ and $yz$ lay in the same face, then $x\sim z$.}
\end{observation}

\begin{lemma}
\label{5N(v)}
Let $G$ be a triangulation with $\Delta(G)=5$ and let $v\in V_3$. Then $N(v)$ contains no vertices of degree $3$ and at most one vertex of degree $4$.  Moreover, if $u\in N(v)\cap V_4$, then $u,v$ have two common neighbours $a,b \in V_5$, and $u$ has one other neighbour $c$, such that $a,b, c$ induces a triangle in $G$. See Figure \ref{V3V4pic}.
\end{lemma}

\begin{figure}[htb]
    \centering
    \includegraphics[height=3cm]{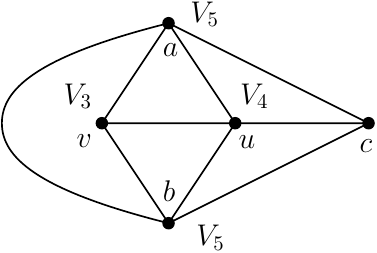}
    \caption{Adjacent vertices $v\in V_3$ and $u\in V_4$, as described in Lemma \ref{5N(v)}.}
    \label{V3V4pic}
  \end{figure}

\begin{proof}
Let $v \in V_3$ and let $G'=G-v$. Since $G$ is a triangulation we know $N(v)$ induces a triangle, so $G'$ is also a triangulation.

Suppose first that $v$ has some neighbor $u\in V_3$ in $G$. This means $deg_{G'}(u)=2$. Since $G'$ is a triangulation it must be that $G'=K_3$. However, this implies that $G=K_4$, which does not have maximum degree 5.

Suppose now that $v$ has some neighbor $u\in V_4$ in $G$. Let $N_G(v)=\{u, a, b\}$. Since $G$ is a triangulation, $N_G(v)$ induces a triangle, and hence $N_G(u)=\{v, a, b, c\}$ for some other vertex $c$. Moreover, since $u$ has degree 4, edges $bu, uc$ (and $au, uc$) are on the same face, meaning that  $ca, cb \in E(G)$ by Observation \ref{samef}. Also, since $v$ has degree 3, edges $av, vb$ are on the same face, and hence $ab\in E(G)$ by Observation \ref{samef}.
So we have precisely the situation pictured in Figure \ref{V3V4pic}, save that we have not yet argued that $a, b \in V_5$. 

Let $G''=G'-u=G-\{v,u\}$. We know $G'$ is a triangulation and that $deg_{G'}(u)=3$. This means $G''$ is also a triangulation. If either of $a$ or $b$, has degree $4$ in $G$, then they have degree 2 in $G''$. This would imply that $G''= K_3$, and hence that $G$ has no vertex of degree 5, contradiction. 
\end{proof}

In any triangulation, the neighbourhood of a vertex $x$ contains a cycle on which its neighbours appear, say in clockwise order, in the embedding around $x$. The following lemma says something about this ordering.

\begin{lemma}\label{nonadj} Let $G$ be a triangulation with $\Delta(G)=5$, and let $u, v$ be non-adjacent vertices in $G$. If $u,v$ have a common neighbour $y$, then they must have a second common neighbour $z$, and moreover $y,z$ must be consecutive neighbours in both the clockwise embedding of $N(u)$ and of $N(v)$.   
\end{lemma}

\begin{proof} If $u, v$ have precisely one common neighbour $y$, then $y$ has degree at least 6, owing to its adjacencies to $u,v$, to its two neighbours on the cycle around $N(u)$, and similarly in $N(v)$. See, for example, the left picture in Figure \ref{V3V4disj}. In fact, even if $u,v$ have a second common neighbour $z$, $y$ will still have degree at least 6, unless $z$ is consecutive to $y$ on both the cycles around $N(u)$ and $N(v)$. See the right picture in Figure \ref{V3V4disj}.
\end{proof}

\begin{figure}[htb]
    \centering
    \includegraphics[height=3cm]{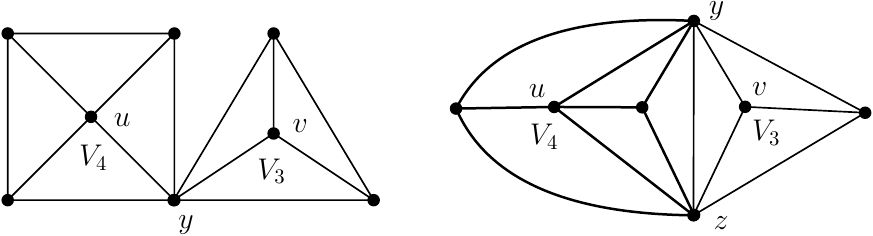}
    \caption{Two situations considered in Lemma \ref{nonadj}.}
    \label{V3V4disj}
  \end{figure}

\begin{lemma}\label{V3V4nonadj} Let $G$ be a triangulation with $\Delta(G)=5$ and let $v\in V_3$, $u\in V_4$ with $v\not\sim u$. Then either $N(u)\cap N(v)=\emptyset$, or $G$ is one of the triangulations $G_6, G_{11}$ pictured in Figure \ref{AllTriangulations}.
\end{lemma}

\begin{proof} If $N(u)\cap N(v)=\emptyset$, then we are done so assume $N(u)\cap N(v) \neq \emptyset$. By Lemma \ref{nonadj}, $|N(u)\cap N(v)|\in \{2,3\}$. Suppose first that $|N(u)\cap N(v)|=2$, say with common neighbors $a,b$ and with $c$ the third neighbor of $v$. We may assume by Lemma \ref{nonadj} that $a\sim b$, $a\sim c$, and $b\sim c$ with $a,b,c$ the clockwise order of the embedding of $N(v)$. Let $x,y$ be the other neighbors of $u$; we may assume by Lemma \ref{nonadj} that $a\sim b$, $a\sim x, x\sim y, y\sim b$ with $a, b, y, x$ the counter-clockwise order of the embedding of $N(u)$. See the left picture in Figure \ref{V3V4nonAdj}. Both $a,b$ have five neighbors among $N[v]\cup N[u]$, and hence edges $ac, ax$ are on the same face, and edges $bc, cy$ are on the same face. So by Observation \ref{samef}, $cx, cy \in E(G)$. However now $c$ also has 5 neighbours among $N[v]\cup N[u]$. Since $c, x, y$ is a triangle and $G$ is a triangulation, this means there are no other vertices or edges in the graph, and indeed the only possible triangulation is $G_{11}$ pictured in Figure \ref{AllTriangulations}.

\begin{figure}[ht]
    \centering
    \includegraphics[height=3cm]{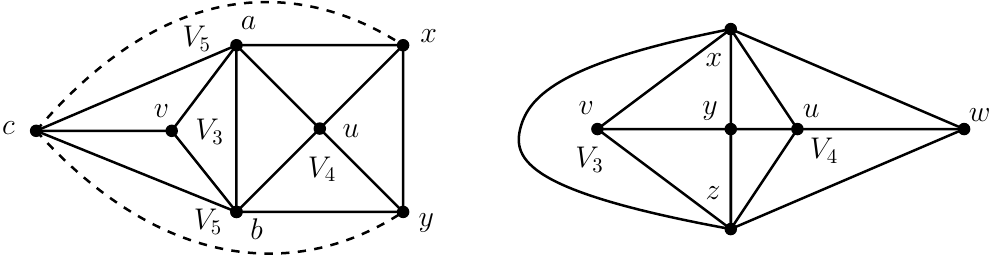}
    \caption{Two situations considered in the proof of Lemma \ref{V3V4nonAdj}.}
    \label{V3V4nonAdj}
  \end{figure}

We may now assume that $|N(u)\cap N(v)|=3$, say with common neighbors $x, y, z$. See the right picture in Figure \ref{V3V4nonAdj}. Then one of these vertices, say $y$, is forced to have degree 4 in $G$, and the other two, say $x,y$, have five neighbors among $N(v)\cap N(u)$. Moreover, if $w$ is the other neighbor of $u$, then $x, z, w$ forms a triangle. Since $G$ is a triangulation, this again means there are no other vertices of edges in the graph, and indeed the only possible triangulation is $G_{6}$ pictured in Figure \ref{AllTriangulations}.
\end{proof}

\begin{proposition}
\label{only8}
If $G$ is a triangulation with $\Delta(G)=5$ then $G$ has a degree sequence corresponding to one of the following eight cases: $4,6,11,13,15,16,17,18$. 
\end{proposition}

\begin{proof} Given our discussion following Table 1 (eliminating cases 8 and 14), and since cases 1, 7, 19 do not involve vertices of degree 5, it remains only for us to eliminate cases 2, 3, 5, 9, 10 and 12 as possible degree sequences.

Lemma \ref{5N(v)} says that every 3-vertex is adjacent to at least two 5-vertices. Both case 3 and case 12 have a 3-vertex, but only one 5-vertex, which makes this impossible. Case 2 can also be eliminated, since there $|V_3|=|V_5|=3$ (and $|V_4|=0$), so Lemma \ref{5N(v)} implies that $G$ contains a copy of $K_{3,3}$. 

We may now assume, for a contradiction, that $G$ has a degree sequence corresponding to case 5, 9, or 10.
In each case we can choose a pair of vertices $v\in V_3$ and $u\in V_4$. 

Suppose first that $v\sim u$. Then we get the situation described by Lemma \ref{5N(v)} and depicted in Figure \ref{V3V4pic} with $a, b, c \in V(G)$. Let $S=\{u, v, a, b, c\}$. Between $S$ and $V(G)\setminus S$ we will have one edge each from $a$ and $b$, and up to two edges from $c$, for at total of at most four edges. 

In case 5, $V(G)\setminus S$ consists of exactly two vertices (either two vertices of degree 5 or one vertex of degree 5 and one vertex of degree 3), and their degree requires at least $4+2=6$ edges from $S$, contradiction. In case 9, $V(G)\setminus S$ consists of 4 vertices, all of which have degree 5, so they need a total of at least $4(5-3)=8$ edges from $S$, contradiction. In case 10, $V(G)\setminus S$ consists of exactly three vertices (either all of degree 5, or two of degree 5 and one of degree 4), and their degree requires at least $3+3+2=8$ edges from $S$, contradiction.

We may now assume that $v\not\sim u$. Since neither of $G_6$ nor $G_{11}$ has a degree sequence corresponding to case 5, 9, or 10, we know by Lemma \ref{V3V4nonadj} that $N(v)\cap N(u)=\emptyset$. This means that $N[v]\cup N[u]$ contains 9 different vertices in $G$, immediately disqualifying cases 5 or 10 (which have a total of 7 and 8 vertices, respectively). Case 9 has exactly 9 vertices, meaning that $N[v]\cup N[u]=V(G)$. Moreover, all vertices besides $u,v$ have degree 5 in this case. We know that $N(u)$ contains a cycle, since $G$ is a triangulation. Each of the four vertices in $N(u)$ must have two additional neighbors, in addition to $u$ and its two neighbors on this cycle. By planarity, this means that the cycle in $N(u)$ is induced. However then, in order to satisfy the degree requirements of vertices in $N(u)$, there must be exactly $4(2)=8$ edges between $N(u)$ and $N(v)$. On the other hand, the vertices of $N(v)$ require exactly $3(2)=6$ edges, contradiction. 
\end{proof}

We can now establish Theorem \ref{Pics}.

\begin{proof} \emph{(Theorem \ref{Pics})}
First suppose that $\Delta(G)<5$. By Lemma \ref{count} this means that either $G=K_3\in \mathcal{G}$ or the degree sequence of of $G$ corresponds to one of the cases $\{1,7,19\}$. If $G$ has the degree sequence of case 1, then it must be $K_4$. If $G$ has the degree sequence of case 7 or case 19, then there is a $v \in V(G)$ such that $\deg(v)=4$. We know $v$ forms a 4-wheel with its neighbors since $G$ is a triangulation. If $G$ has the degree sequence of case 7, then there are no vertices outside this 4-wheel and joining two nonconsecutive neighbors of $v$ will yield $G_7\in\mathcal{G}$. If $G$ has the degree sequence of case 19, then there is 1 vertex outside the 4-wheel which must be joined to all neighbors of $v$ to yield $G_{19}\in\mathcal{G}$.

We may now assume that $G$ is a  triangulation with $\Delta(G)=5$; by Preposition \ref{only8} this means the degree sequence of $G$ corresponds to one of the following eight cases: $4,6,11,13,15,16,17,18$. It is already known that the icosohedron (which is $G_{13}\in\mathcal{G}$) is the  unique embedding of a 5-regular triangulation, and we won't repeat this argument here.

Suppose that $G$ has the degree sequence of case 6 or case 11. Then we have $|V_3|\geq 1$, $|V_4|\geq2$, and $v(G)\in\{6,7\}$. Let $v\in V_3$; by Lemma \ref{5N(v)} we know there is a $u\in V_4$ such that $u\not\sim v$. In particular, this means that we cannot have $N(v)\cap N(u)=\emptyset$ (since that would require at least 4+5=9 vertices in $G$). So by Lemma \ref{V3V4nonadj}, $G \in\{G_6, G_{11}\}\subseteq \mathcal{G}$ .

Now suppose that $G$ has the degree sequence of case 4, 15, 16, 17, or 18. Note that $|V(G)|\in\{7, 8, 9, 10\}$. Choose $w\in V_5$ and let $w_1, \ldots, w_5$ be the clockwise neighbours of $w$, appearing along a cycle $C$. Since $|V(G)|>6$ there exists $u\in V(G)$ with $u\not\sim w$. If $u, w$ have no common neighbours then $N[u]$ accounts for an extra 4 to 6 vertices in the graph. However, this is only possible with $|V(G)|=10$ and $\deg(u)=3$. These two values cannot happen simultaneously since the only case with $V_3\neq 0$ is case 4, and it has only 8 total vertices. So $u$ and $w$ must have common neighbors. By Lemma \ref{nonadj} we may assume that, without loss of generality, $u\sim w_1, w_2$, and moreover, $N(w)\cap N(u)=\{w_1, w_2, \ldots, w_i\}$ for some $2\leq i\leq 5$. 

Suppose first that $i=5$. Then $u, w$ both have degree 5, $w_1, \ldots, w_5$ have degree 4, and $G$ must look like the graph $G_{18}\in\mathcal{G}$.

We may now assume that $2\leq i\leq 4$. Then $w_1$ and $w_i$ have degree 5 in $N[w]\cap N[u]$ due  to their adjacencies to $u$, $w$ and to the two vertices consecutive to them around both $N(u)$ and $N(w)$. We also know $w_{2}, \ldots, w_{i-1}$ (which is possibly an empty list), all have degree 4. See, for example, Figure \ref{w1wi}.

\begin{figure}[hb]
    \centering
    \includegraphics[height=2.2cm]{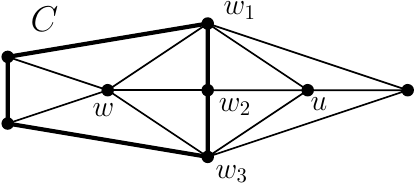}
    \caption{An example where $i=3$.}
    \label{w1wi}
  \end{figure}

Suppose now that we are in case 4. Since $V_4=\emptyset$, we know that $i=2$. We claim that we can choose our afore-mentioned $u\not\sim w$ so that $u\in V_3$. The only issue with this would be if $C$ contains both $V_3$ vertices, which we claim cannot happen. Otherwise, let $U$ be the two vertices of $G$ not in $N[w]$. We would have at most $6$ edges between $C$ and $U$ (at most two from each of the three 5-vertices on $C$). On the other hand, $U$ consists of two 5-vertices in this case, meaning that there are at least $8$ edges between $U$ and $C$, contradiction. So, we may indeed choose $u\not\sim w$ with $u\in V_3$.  Let $N(u)=\{w_1, w_2, x\}$ for some $x\in V(G)$. See Figure \ref{G4casePlus}(a). Since $w_1$ has degree 5, $w_5w_1$ and $w_1x$ lie on the same face, so by Observation \ref{samef}, $w_5 f \in E(G)$. Similarly, $w_3x\in E(G)$. Since now $x$ has degree 5 and $w_5, w_3$ have degree at least 4, it must be that $w_4$ is the other vertex of degree 3 in $G$. But then $w_3w_4, w_4w_5$ lie on the same face, so $w_3w_5 \in E(G)$, and $G$ must look like the graph $G_4\in\mathcal{G}$.

\begin{figure}[h]
    \centering
    \includegraphics[height=14cm]{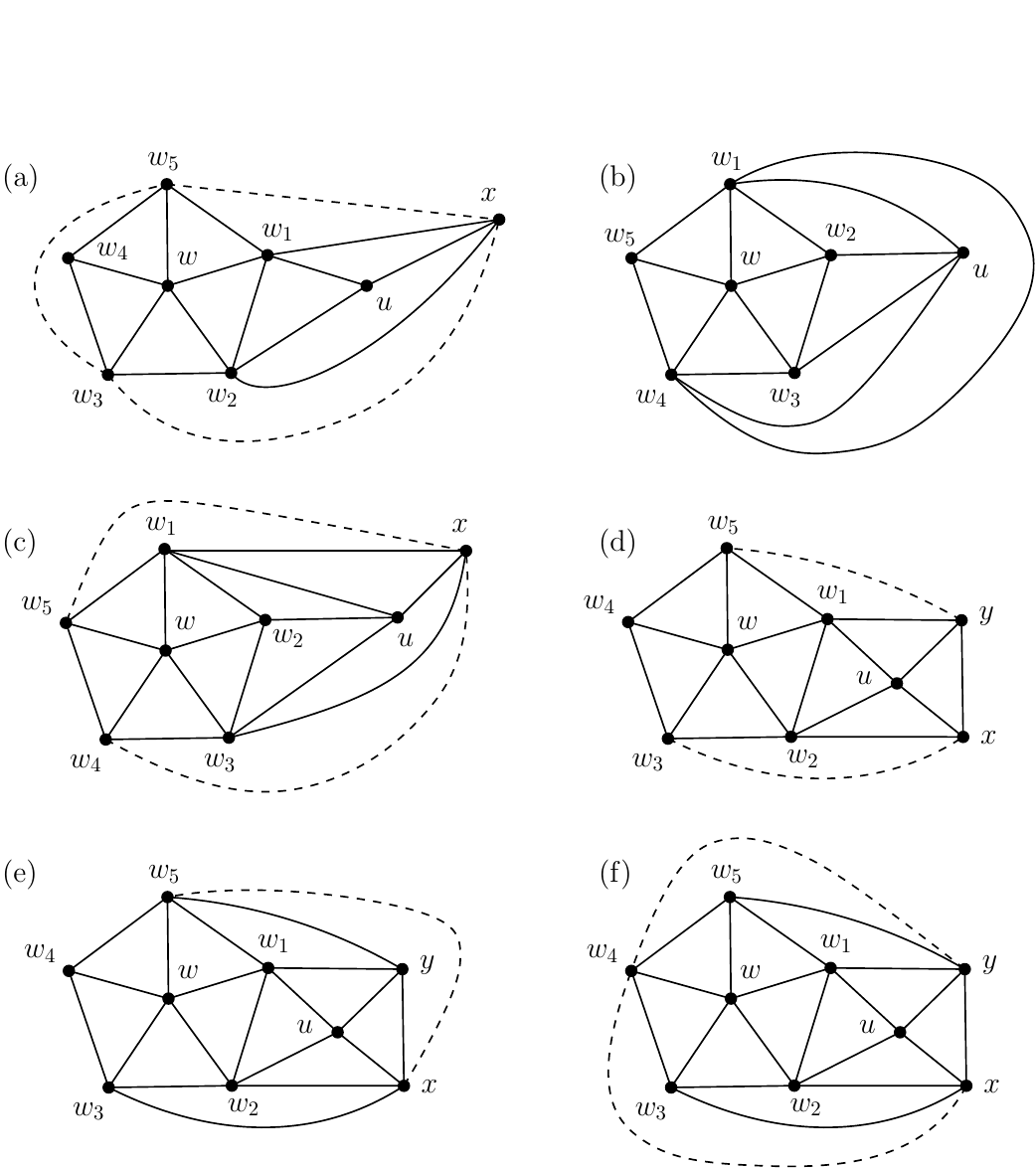}
    \caption{Various situations considered in the proof of Theorem \ref{Pics}.}
    \label{G4casePlus}
  \end{figure}

We may now assume we are in case 15, 16, 17, or 18. We claim that either we may choose our previously discussed $u\not\sim w$ so that $u\in V_4$, or that we have $G=G_{18}$. Since each vertex of degree 4 can be adjacent to at most four vertices of degree 5, having $(|V_4|, |V_5|)$ be $(2,8)$ or $(3,6)$ (as in Cases 15 and 16) means this choice of $u$ is certainly possible. Having this value be $(4,4)$ as in case 17 also means this is possible, since while two vertices of degree 4 may share an identical neighborhood of 4 vertices of degree 5, planarity will not allow for our vertices of degree 4 to have such identical neighborhoods. In Case 18 it could be possible for the two 5-vertices to have identical neighborhoods, consisting of five 4-vertices, but then we get exactly $G_{18}$. So we indeed have our claim.

We may now choose the vertex $u\not\sim w$ with $u\in V_4$. So $i\in\{2, 3, 4\}$. Suppose first that $i=4$. Then $w_1, w_4$ have degree 5, so $w_4w_5, w_5w_1$ are on the same face, and hence $w_5$ has degree 3. See Figure \ref{G4casePlus}(b). But this contradicts the fact that we are in one of cases 15, 16, 17, or 18, where $V_3=\emptyset$. Suppose next that $i=3$. Let $N(u)=\{w_1, w_2, w_3, x\}$ for some $x\in V(G)$. See Figure \ref{G4casePlus}(c). Then $w_1$ has degree 5, so $w_5w_1, w_1x$ are on the same face, meaning $w_5x\in E(G)$. Similarly, since $w_4$ has degree 5, $w_4x\in E(G)$. We now have $x$ of degree 5, and $x, w_5, w_4$ forming a triangle in $G$. Hence $V(G)=N[u]\cup N[v]$ and $G$ looks exactly like $G_{17}\in\mathcal{G}$.

We may now assume that $i=2$. Let $N(u)=\{w_1, w_2, x, y\}$ with this being the counter-clockwise order around $u$. Then $w_1, w_2$ have degree 5, so $yw_5, xw_3\in E(G)$. See Figure \ref{G4casePlus}(d). Note that $|N[u]\cup N[w]|=9$, meaning that we must be in either case 15 (10 vertices) or case 16 (9 vertices). 

Suppose first that we are in case 15, meaning there exists one more vertex $z\in V(G)\setminus N[u]\cup N[w]$. Note that, by our above arguments, we may in fact assume that all non-adjacent 4-vertices and 5-vertices in $G$ share exactly two common neighbors. This means that if $z$ has degree 4, it must be adjacent to two neighbours of $w_1$ and two neighbours of $w_2$, meaning its neighbour set is precisely $w_5,y, x, w_3$. But then this forces $w_4$ to be degree 3 in $G$, contradiction. Hence $z$ must have degree 5, with neighbors  $w_5, y, x, w_3,  w_4$. But then $G$ must be exactly the graph $G_{15}\in\mathcal{G}$.

Suppose finally that we are in case 16, meaning that $V(G)=N[u]\cup N[w]$.  Since there are no degree 3 vertices in case 16, we may assume wlog that $x\sim w_4$. See Figure \ref{G4casePlus}(f). Since there are exactly three 4-vertices in case 16, exactly one of $w_4, w_5, y$ must have degree 4. Since $w_5$ is already adjacent to $w_4, y$, it must be $w_5$ that has degree 4, and $w_4, y$ must be adjacent. Hence $G$ looks exactly like $G_{16}\in\mathcal{G}$.
\end{proof}

\section{List-edge-coloring the triangulations}


Let $G$ be a graph with $E(G)=\{e_1, \ldots, e_m\}$. Define a polynomial $p_G$ on the variables $x_i$ for $1\leq i \leq m$ as $$p_G=\prod_{e_i \sim e_j, i<j} (x_i-x_j),$$
where by $e_i \sim e_j$ we mean that these edges are adjacent in $G$.
Given some $p_G$ and some $f:E(G)\rightarrow \mathbb{N}$, by $[x^f]p_G$ we mean the coefficient of the monomial in $p_G$ where $x_i$ has exponent $f(e_i)$, for $1\leq i\leq m$. Given some $e\in E(G)$, by $[x^{f+1_e}]p_G$ we mean the same thing as $[x^f]p_G$ except with the exponent of the variable corresponding to $e$ increased by one. These quantities are significiant when considering the list-edge-colorability of a graph, as evidenced by the following theorems.

\begin{theorem}[Alon and Tarsi \cite{AT}]\label{AT}
    Let $G$ be a graph and let $f:E(G)\rightarrow \{0, 1, \ldots, k-1\}$. If $[x^f]p_G\neq 0$, then $G$ is $k$-list-edge-colorable.
\end{theorem}

\begin{theorem}[Dvo\v{r}\'{a}k \cite{Dv}]\label{extAT}
    Let $G$ be a graph, let $L$ be an edge-list-assignment with $|L_e|=k$ for all $e\in E(G)$, and let $f:E(G)\rightarrow \{0, 1, \ldots, k-1\}$ be such that $\sum_{e\in E(G)} f(e)=\sum_{v\in V(G)}{\deg(v) \choose 2}-1$. If $G$ is not $L$-edge-colorable, then for every color c, $$\sum_{e\in E(G)} [x^{f+1_e}]p_G \cdot \chi_{L, c}(e)= 0,$$
    where $\chi_{L, c}(e)$ is 1 if $c\in L_e$ and 0 otherwise.
\end{theorem}

It is worth noting that Theorems \ref{AT} and \ref{extAT} are simplified  versions of what appeared in \cite{AT} and \cite{Dv}, tailored to suit our needs here. In particular, these are the line-graph versions of more general theorems about vertex-coloring, and in the more general versions the lists do not need to have a constant size. Despite this simplification, there are two fundamental challenges in applying Theorem \ref{AT} successfully: one must actually compute the relevant coefficients of the polynomial in question and, if all such coefficients are zero, we do not get any information on whether the graph is $k$-list-edge-colorable or not. In \cite{Dv}, Dvo\v{r}\'{a}k helps on both of these fronts by providing an algorithm to find the coefficients of Theorem \ref{AT} that is more efficient in practice than direct computation which also incorporates Theorem \ref{extAT} when needed.

Note that in order to get $[x^f]p_G\neq 0$ in Theorem \ref{AT}, we must have $\sum_{e\in E(G)}f(e)=\sum_{v\in V(G)}{\deg(v) \choose 2}$, since every term in $p_G$ will have this same total degree. However, if all such functions $f$ result in $[x^f]p_G= 0$ (i.e. Theorem \ref{AT} gives no information), then Theorem \ref{extAT} tells us to look at $f$ with total sum one less to potentially get valuable information. In this situation each monomial considered by $[x^{f+1_e}]p_G$ will have the same total degree as before, but one exponent may be as large as $k$. Theorem \ref{extAT} uses the coefficients of these monomials to produce a system of linear equations in the variables $\chi_{L, c}(e)$ for all $e\in E(G)$ (with one equation for each appropriate $f$), which can be used to determine information about a potentially ``bad''  list-assignment $L$. Such potentially bad $L$ would then be output by Dvo\v{r}\'{a}k's algorithm. Of course, if there are only a small number of possibly bad $L$ output by the algorithm, then it may be possible to directly check these to determine whether the graph is $k$-list-edge-colorable or not. In the end, we will see that this outcome does not happen when we apply Dvo\v{r}\'{a}k's algorithm to the particular graphs we are interested in.

\setcounter{theorem}{2}
\begin{theorem} If $G$ is a triangulation with $\Delta \leq 5$, then $\chi'_l(G)=\chi'(G).$
\end{theorem}

\begin{proof} By Theorem \ref{Pics}, $\mathcal{G}$ is the set of all triangulations with $\Delta\leq 5$. Some of these graphs are already known to satisfy the LECC, in particular $G_0=K_3$, $G_1=K_4$, and  the icosahedron, $G_{13}$ (Ellingham and Goddyn \cite{EG}). Since $\chi'_\ell(K_6)=5$ is known (Cariolario, Cariolario, Schauz, and Sun \cite{CCSS}) and $G_6\subset K_6$ with $\chi'(G_G)\geq \Delta(G_6)=5$, we immediately get that $G_6$ satisfies the LECC.

It remains now for us to verify the LECC for the graphs in $\mathcal{G}\setminus\{G_0, G_1, G_6, G_{13}\}$. Of these, only two ($G_7$ and $G_{19}$) have maximum degree $4$ (see Figure \ref{AllTriangulations}), and we will show they are 4-list-edge-colorable. The rest of the graphs have maximum degree 5, and we will show they are all 5-list-edge-colorable, except for $G_{16}$. Since $G_{16}$ has 9 vertices and 21 edges, $\chi(G_{16})\geq \frac{21\cdot 2}{8}>5$; we will show that $G_{16}$ is 6-list-edge-colorable.

The two graphs we plan to show are 4-list-edge-colorable, $G_7$ and $G_{19}$, have the property that $G_{7}\subset G_{19}$ (see Figure \ref{AllTriangulations}). Therefore it suffices to just verify that $G_{19}$ is 4-list-edge-colorable. Similarly, we claim that two of the graphs we plan to show are 5-list-edge-colorable, $G_4$ and $G_{11}$, satisfy $G_{11}\subset G_4$. To see this, first note that $G_4$ has $|V_3|=2$ and $|V_5|=6$ (see Table 1) and its vertices of degree 3 are non-adjacent (see Figure 1). By deleting either vertex of degree 3 in $G_4$, we get another triangulation with $|V_3|=1, |V_4|=3, |V_5|=3$, which by Theorem \ref{Pics} must be $G_{11}$.  Therefore in order to show that $G_4$ and $G_{11}$ are 5-list-edge-colorable, it suffices to just check $G_{4}$.

We now apply Dvo\v{r}\'{a}k's afore-mentioned algorithm to $G_i$ for each $i\in\{4, 15,16,17,18,19\}$. This algorithm is available at 
\begin{center}\url{https://gitlab.mff.cuni.cz/dvorz9am/alon-tarsi-method},\end{center} with a link also included in the abstract of \cite{Dv}. The algorithm can be complied using g++ on Linux. Since the algorithm is intended to consider list-vertex-coloring, we must take the line graph of each of $G_i$ for each $i\in\{4, 15,16,17,18,19\}$ before inputting them into the algorithm. For this prepossessing step, we label the vertices of our graphs as shown in Figure \ref{LabelledTriangulations} and use Algorithm 1 printed in Appendix A. This algorithm is run in Sage by taking each of our labeled graphs and our desired list-size for each graph as an input. The output of Algorithm 1 is the required line graph input for Dvo\v{r}\'{a}k's algorithm, in the correct format, so we can apply it at that point. It is worth noting that we used Dvo\v{r}\'{a}k's code verbatim except for changing his parameter MAX\_VERTICES from 21 to 24 to accommodate our larger graphs.

\begin{figure}[htb]
    \centering
    \includegraphics[height=7cm]{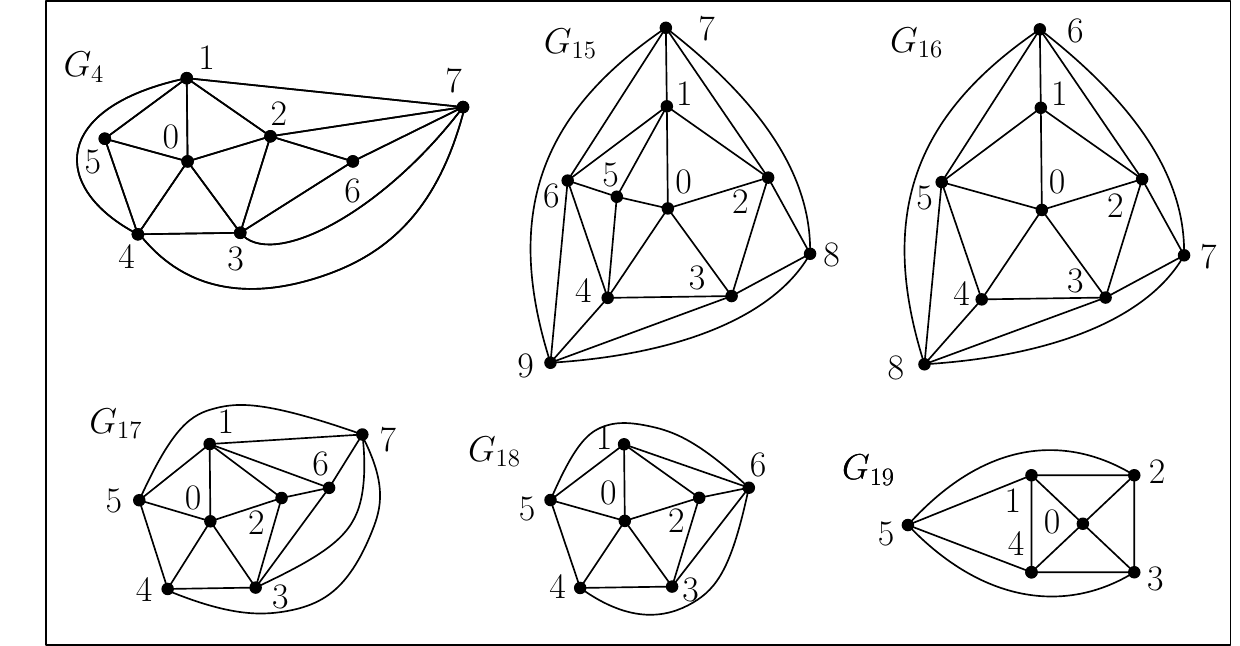}
    \caption{Some labelled triangulations from $\mathcal{G}$.}
    \label{LabelledTriangulations}
  \end{figure}

For each $G_i$ with $i\in\{4,15,16,17,18,19\}$, the input and output of Algorithm 1, as well as the printed output of Dvo\v{r}\'{a}k's algorithm, are listed a Computations \ref{comp4} through \ref{comp19} in Appendix B. In the output of Algorithm 1 (i.e.~input to Dvo\v{r}\'{a}k's algorithm), note that we have used the endline symbol to denote an actual line break, just for space efficiency in this paper. In each of the computations, the output of  Dvo\v{r}\'{a}k's algorithm is: ``Always colorable (using standard Alon-Tarsi)''.
\end{proof}

As a final note, we remark that Dvo\v{r}\'{a}k's algorithm has succeeded where two other approaches failed. When trying to determine if the graphs $\mathcal{G}$ satisfied the list-edge-coloring conjecture, the present authors first employed the so-called ``kernel method,"  initiated by Galvin \cite{G} and generalized by Borodin, Kostochka, and Woodall (\cite{BKW} and \cite{bkw2}). Included in the first author's PhD thesis \cite{jh} are the details of using this method to show that $G_{11}$, $G_{17}$, and $G_{18}$ satisfy Vizing's conjecture, but this method did not work for $G_4,G_{15},G_{16}, \text{and } G_{19}$ even to verify this weaker version of LECC. Also included in this dissertation are computations that show that some of these graphs satisfy the LECC. These computations use an implementation by Schauz \cite{S}, based on his extension \cite{s2} of Theorem \ref{AT}.  Schauz's algorithm is for regular graphs only, but by embedding our graphs in regular graphs, one can show that the LECC holds for $G_4,G_{11},G_{16},G_{17}, \text{and } G_{19}$. Schausz's algorithm was also able to show that $G_{15}$ and $G_{18}$ satisfy Vizing's conjecture, but it fell short to show they satisfy LECC.

\section*{Acknowledgements}

We thank Jonathan Joe at Middle Georgia State University for his computing expertise. We also thank an anonymous referee for their insightful comments.

\appendix
\section{Algorithm 1}


\vspace{.5cm}

\noindent {\large \textbf{Algorithm 1}}
\small
\begin{verbatim}
edges = [(0,1),(0,2),(0,3),(0,4),(1,2),(2,3),(3,4),(1,4),(1,5),(2,5),(3,5),(4,5)]
listsize = 5

G = Graph(edges)
edges_dict = {}
counter = 0
for edge in G.edges(labels=False):
    edges_dict[edge] = counter
    counter += 1

H=G.line_graph(labels=False)
H.edges(labels=False)

h_prime_edges = []
for edge in H.edges(labels=False):
    v_1 = edge[0]
    v_2 = edge[1]
    n_1 = edges_dict[v_1]
    n_2 = edges_dict[v_2]
    h_prime_edges.append((n_1,n_2))

H_prime = Graph(h_prime_edges)
print(len(H_prime.vertices()), len(H_prime.edges()))
M=[listsize for n in range(0,len(H_prime.vertices()))]
print(*M)
for edge in H_prime.edges(labels=False):
    print('{} {}'.format(edge[0],edge[1]))
\end{verbatim}

\normalsize

\normalsize


\section{Computations}

\begin{computation}[$G_4$]
\label{comp4}
\emph{\phantom{m}\\
Algorithm 1 edges: [(0,1),(0,2),(0,3),(0,4),(0,5),(1,2),(2,3),(3,4),(4,5),(1,5),(2,6),(3,6),(6,7),(2,7),(3,7),(1,7),\newline (4,7)]\\
Algorithm 1 listsize: 5\\
Algorithm 1 output:17 58 \textbackslash \textbackslash
5 5 5 5 5 5 5 5 5 5 5 5 5 5 5 5 5 \textbackslash \textbackslash
0 1 \textbackslash \textbackslash
0 2 \textbackslash \textbackslash
0 3 \textbackslash \textbackslash
0 4 \textbackslash \textbackslash
0 5 \textbackslash \textbackslash
0 6 \textbackslash \textbackslash
0 7 \textbackslash \textbackslash
1 2 \textbackslash \textbackslash
1 3 \textbackslash \textbackslash
1 4 \textbackslash \textbackslash
1 5 \textbackslash \textbackslash
1 8 \textbackslash \textbackslash
1 9 \textbackslash \textbackslash
1 10 \textbackslash \textbackslash
2 3 \textbackslash \textbackslash
2 4 \textbackslash \textbackslash
2 8 \textbackslash \textbackslash
2 11 \textbackslash \textbackslash
2 12 \textbackslash \textbackslash
2 13 \textbackslash \textbackslash
3 4 \textbackslash \textbackslash
3 11 \textbackslash \textbackslash
3 14 \textbackslash \textbackslash
3 15 \textbackslash \textbackslash
4 6 \textbackslash \textbackslash
4 14 \textbackslash \textbackslash
5 6 \textbackslash \textbackslash
5 7 \textbackslash \textbackslash
5 8 \textbackslash \textbackslash
5 9 \textbackslash \textbackslash
5 10 \textbackslash \textbackslash
6 7 \textbackslash \textbackslash
6 14 \textbackslash \textbackslash
7 10 \textbackslash \textbackslash
7 13 \textbackslash \textbackslash
7 15 \textbackslash \textbackslash
7 16 \textbackslash \textbackslash
8 9 \textbackslash \textbackslash
8 10 \textbackslash \textbackslash
8 11 \textbackslash \textbackslash
8 12 \textbackslash \textbackslash
8 13 \textbackslash \textbackslash
9 10 \textbackslash \textbackslash
9 12 \textbackslash \textbackslash
9 16 \textbackslash \textbackslash
10 13 \textbackslash \textbackslash
10 15 \textbackslash \textbackslash
10 16 \textbackslash \textbackslash
11 12 \textbackslash \textbackslash
11 13 \textbackslash \textbackslash
11 14 \textbackslash \textbackslash
11 15 \textbackslash \textbackslash
12 13 \textbackslash \textbackslash
12 16 \textbackslash \textbackslash
13 15 \textbackslash \textbackslash
13 16 \textbackslash \textbackslash
14 15 \textbackslash \textbackslash
15 16 \textbackslash \textbackslash\\
Dvo\v{r}\'{a}k output:Always colorable (using standard Alon-Tarsi)\\
}
\end{computation}

\begin{computation}[$G_{15}$]
\label{comp15}
\emph{\phantom{m}\\
Algorithm 1 edges: [(0,1),(0,2),(0,3),(0,4),(0,5),(1,2),(2,3),(3,4),(4,5),(1,5),(1,6),(4,6),(5,6),(6,7),(6,9),(1,7),\newline (2,7),(7,8),(7,9),(2,8),(3,8),(8,9),(3,9),(4,9)]\\
Algorithm 1 listsize: 5\\
Algorithm 1 output:24 92 \textbackslash \textbackslash
5 5 5 5 5 5 5 5 5 5 5 5 5 5 5 5 5 5 5 5 5 5 5 5 \textbackslash \textbackslash
0 1 \textbackslash \textbackslash
0 2 \textbackslash \textbackslash
0 3 \textbackslash \textbackslash
0 4 \textbackslash \textbackslash
0 5 \textbackslash \textbackslash
0 6 \textbackslash \textbackslash
0 7 \textbackslash \textbackslash
0 8 \textbackslash \textbackslash
1 2 \textbackslash \textbackslash
1 3 \textbackslash \textbackslash
1 4 \textbackslash \textbackslash
1 5 \textbackslash \textbackslash
1 9 \textbackslash \textbackslash
1 10 \textbackslash \textbackslash
1 11 \textbackslash \textbackslash
2 3 \textbackslash \textbackslash
2 4 \textbackslash \textbackslash
2 9 \textbackslash \textbackslash
2 12 \textbackslash \textbackslash
2 13 \textbackslash \textbackslash
2 14 \textbackslash \textbackslash
3 4 \textbackslash \textbackslash
3 12 \textbackslash \textbackslash
3 15 \textbackslash \textbackslash
3 16 \textbackslash \textbackslash
3 17 \textbackslash \textbackslash
4 6 \textbackslash \textbackslash
4 15 \textbackslash \textbackslash
4 18 \textbackslash \textbackslash
5 6 \textbackslash \textbackslash
5 7 \textbackslash \textbackslash
5 8 \textbackslash \textbackslash
5 9 \textbackslash \textbackslash
5 10 \textbackslash \textbackslash
5 11 \textbackslash \textbackslash
6 7 \textbackslash \textbackslash
6 8 \textbackslash \textbackslash
6 15 \textbackslash \textbackslash
6 18 \textbackslash \textbackslash
7 8 \textbackslash \textbackslash
7 16 \textbackslash \textbackslash
7 18 \textbackslash \textbackslash
7 19 \textbackslash \textbackslash
7 20 \textbackslash \textbackslash
8 10 \textbackslash \textbackslash
8 19 \textbackslash \textbackslash
8 21 \textbackslash \textbackslash
8 22 \textbackslash \textbackslash
9 10 \textbackslash \textbackslash
9 11 \textbackslash \textbackslash
9 12 \textbackslash \textbackslash
9 13 \textbackslash \textbackslash
9 14 \textbackslash \textbackslash
10 11 \textbackslash \textbackslash
10 19 \textbackslash \textbackslash
10 21 \textbackslash \textbackslash
10 22 \textbackslash \textbackslash
11 13 \textbackslash \textbackslash
11 21 \textbackslash \textbackslash
11 23 \textbackslash \textbackslash
12 13 \textbackslash \textbackslash
12 14 \textbackslash \textbackslash
12 15 \textbackslash \textbackslash
12 16 \textbackslash \textbackslash
12 17 \textbackslash \textbackslash
13 14 \textbackslash \textbackslash
13 21 \textbackslash \textbackslash
13 23 \textbackslash \textbackslash
14 17 \textbackslash \textbackslash
14 20 \textbackslash \textbackslash
14 22 \textbackslash \textbackslash
14 23 \textbackslash \textbackslash
15 16 \textbackslash \textbackslash
15 17 \textbackslash \textbackslash
15 18 \textbackslash \textbackslash
16 17 \textbackslash \textbackslash
16 18 \textbackslash \textbackslash
16 19 \textbackslash \textbackslash
16 20 \textbackslash \textbackslash
17 20 \textbackslash \textbackslash
17 22 \textbackslash \textbackslash
17 23 \textbackslash \textbackslash
18 19 \textbackslash \textbackslash
18 20 \textbackslash \textbackslash
19 20 \textbackslash \textbackslash
19 21 \textbackslash \textbackslash
19 22 \textbackslash \textbackslash
20 22 \textbackslash \textbackslash
20 23 \textbackslash \textbackslash
21 22 \textbackslash \textbackslash
21 23 \textbackslash \textbackslash
22 23 \textbackslash \textbackslash\\
Dvo\v{r}\'{a}k output:Always colorable (using standard Alon-Tarsi)\\
}
\end{computation}

\begin{computation}[$G_{16}$]
\label{comp16}
\emph{\phantom{m}\\
Algorithm 1 edges:	[(0,1),(0,2),(0,3),(0,4),(0,5),(1,2),(2,3),(3,4),(4,5),(1,5),(5,6),(1,6),(2,6),(2,7),(3,7),(3,8),\newline (4,8),(5,8),(6,7),(7,8),(6,8)]\\
Algorithm 1 listsize: 6\\
Algorithm 1 output:21 78 \textbackslash \textbackslash
6 6 6 6 6 6 6 6 6 6 6 6 6 6 6 6 6 6 6 6 6 \textbackslash \textbackslash
0 1 \textbackslash \textbackslash
0 2 \textbackslash \textbackslash
0 3 \textbackslash \textbackslash
0 4 \textbackslash \textbackslash
0 5 \textbackslash \textbackslash
0 6 \textbackslash \textbackslash
0 7 \textbackslash \textbackslash
1 2 \textbackslash \textbackslash
1 3 \textbackslash \textbackslash
1 4 \textbackslash \textbackslash
1 5 \textbackslash \textbackslash
1 8 \textbackslash \textbackslash
1 9 \textbackslash \textbackslash
1 10 \textbackslash \textbackslash
2 3 \textbackslash \textbackslash
2 4 \textbackslash \textbackslash
2 8 \textbackslash \textbackslash
2 11 \textbackslash \textbackslash
2 12 \textbackslash \textbackslash
2 13 \textbackslash \textbackslash
3 4 \textbackslash \textbackslash
3 11 \textbackslash \textbackslash
3 14 \textbackslash \textbackslash
3 15 \textbackslash \textbackslash
4 6 \textbackslash \textbackslash
4 14 \textbackslash \textbackslash
4 16 \textbackslash \textbackslash
4 17 \textbackslash \textbackslash
5 6 \textbackslash \textbackslash
5 7 \textbackslash \textbackslash
5 8 \textbackslash \textbackslash
5 9 \textbackslash \textbackslash
5 10 \textbackslash \textbackslash
6 7 \textbackslash \textbackslash
6 14 \textbackslash \textbackslash
6 16 \textbackslash \textbackslash
6 17 \textbackslash \textbackslash
7 9 \textbackslash \textbackslash
7 16 \textbackslash \textbackslash
7 18 \textbackslash \textbackslash
7 19 \textbackslash \textbackslash
8 9 \textbackslash \textbackslash
8 10 \textbackslash \textbackslash
8 11 \textbackslash \textbackslash
8 12 \textbackslash \textbackslash
8 13 \textbackslash \textbackslash
9 10 \textbackslash \textbackslash
9 16 \textbackslash \textbackslash
9 18 \textbackslash \textbackslash
9 19 \textbackslash \textbackslash
10 12 \textbackslash \textbackslash
10 18 \textbackslash \textbackslash
10 20 \textbackslash \textbackslash
11 12 \textbackslash \textbackslash
11 13 \textbackslash \textbackslash
11 14 \textbackslash \textbackslash
11 15 \textbackslash \textbackslash
12 13 \textbackslash \textbackslash
12 18 \textbackslash \textbackslash
12 20 \textbackslash \textbackslash
13 15 \textbackslash \textbackslash
13 17 \textbackslash \textbackslash
13 19 \textbackslash \textbackslash
13 20 \textbackslash \textbackslash
14 15 \textbackslash \textbackslash
14 16 \textbackslash \textbackslash
14 17 \textbackslash \textbackslash
15 17 \textbackslash \textbackslash
15 19 \textbackslash \textbackslash
15 20 \textbackslash \textbackslash
16 17 \textbackslash \textbackslash
16 18 \textbackslash \textbackslash
16 19 \textbackslash \textbackslash
17 19 \textbackslash \textbackslash
17 20 \textbackslash \textbackslash
18 19 \textbackslash \textbackslash
18 20 \textbackslash \textbackslash
19 20 \textbackslash \textbackslash\\
Dvo\v{r}\'{a}k output:Always colorable (using standard Alon-Tarsi)\\
}
\end{computation}

\begin{computation}[$G_{17}$]
\label{comp17}
\emph{\phantom{m}\\
Algorithm 1 edges:	[(0,1),(0,2),(0,3),(0,4),(0,5),(1,2),(2,3),(3,4),(4,5),(1,5),(1,6),(2,6),(3,6),(1,7),(3,7),(4,7),\newline (5,7),(6,7)]\\
Algorithm 1 listsize: 5\\
Algorithm 1 output:18 64 \textbackslash \textbackslash
5 5 5 5 5 5 5 5 5 5 5 5 5 5 5 5 5 5 \textbackslash \textbackslash
0 1 \textbackslash \textbackslash
0 2 \textbackslash \textbackslash
0 3 \textbackslash \textbackslash
0 4 \textbackslash \textbackslash
0 5 \textbackslash \textbackslash
0 6 \textbackslash \textbackslash
0 7 \textbackslash \textbackslash
0 8 \textbackslash \textbackslash
1 2 \textbackslash \textbackslash
1 3 \textbackslash \textbackslash
1 4 \textbackslash \textbackslash
1 5 \textbackslash \textbackslash
1 9 \textbackslash \textbackslash
1 10 \textbackslash \textbackslash
2 3 \textbackslash \textbackslash
2 4 \textbackslash \textbackslash
2 9 \textbackslash \textbackslash
2 11 \textbackslash \textbackslash
2 12 \textbackslash \textbackslash
2 13 \textbackslash \textbackslash
3 4 \textbackslash \textbackslash
3 11 \textbackslash \textbackslash
3 14 \textbackslash \textbackslash
3 15 \textbackslash \textbackslash
4 6 \textbackslash \textbackslash
4 14 \textbackslash \textbackslash
4 16 \textbackslash \textbackslash
5 6 \textbackslash \textbackslash
5 7 \textbackslash \textbackslash
5 8 \textbackslash \textbackslash
5 9 \textbackslash \textbackslash
5 10 \textbackslash \textbackslash
6 7 \textbackslash \textbackslash
6 8 \textbackslash \textbackslash
6 14 \textbackslash \textbackslash
6 16 \textbackslash \textbackslash
7 8 \textbackslash \textbackslash
7 10 \textbackslash \textbackslash
7 12 \textbackslash \textbackslash
7 17 \textbackslash \textbackslash
8 13 \textbackslash \textbackslash
8 15 \textbackslash \textbackslash
8 16 \textbackslash \textbackslash
8 17 \textbackslash \textbackslash
9 10 \textbackslash \textbackslash
9 11 \textbackslash \textbackslash
9 12 \textbackslash \textbackslash
9 13 \textbackslash \textbackslash
10 12 \textbackslash \textbackslash
10 17 \textbackslash \textbackslash
11 12 \textbackslash \textbackslash
11 13 \textbackslash \textbackslash
11 14 \textbackslash \textbackslash
11 15 \textbackslash \textbackslash
12 13 \textbackslash \textbackslash
12 17 \textbackslash \textbackslash
13 15 \textbackslash \textbackslash
13 16 \textbackslash \textbackslash
13 17 \textbackslash \textbackslash
14 15 \textbackslash \textbackslash
14 16 \textbackslash \textbackslash
15 16 \textbackslash \textbackslash
15 17 \textbackslash \textbackslash
16 17 \textbackslash \textbackslash\\
Dvo\v{r}\'{a}k output:Always colorable (using standard Alon-Tarsi)\\
}
\end{computation}

\begin{computation}[$G_{18}$]
\label{comp18}
\emph{\phantom{m}\\
Algorithm 1 edges: [(0,1),(0,2),(0,3),(0,4),(0,5),(1,2),(2,3),(3,4),(4,5),(1,5),(1,6),(2,6),(3,6),(4,6),(5,6)]\\
Algorithm 1 listsize: 5\\
Algorithm 1 output:15 50 \textbackslash \textbackslash
5 5 5 5 5 5 5 5 5 5 5 5 5 5 5 \textbackslash \textbackslash
0 1 \textbackslash \textbackslash
0 2 \textbackslash \textbackslash
0 3 \textbackslash \textbackslash
0 4 \textbackslash \textbackslash
0 5 \textbackslash \textbackslash
0 6 \textbackslash \textbackslash
0 7 \textbackslash \textbackslash
1 2 \textbackslash \textbackslash
1 3 \textbackslash \textbackslash
1 4 \textbackslash \textbackslash
1 5 \textbackslash \textbackslash
1 8 \textbackslash \textbackslash
1 9 \textbackslash \textbackslash
2 3 \textbackslash \textbackslash
2 4 \textbackslash \textbackslash
2 8 \textbackslash \textbackslash
2 10 \textbackslash \textbackslash
2 11 \textbackslash \textbackslash
3 4 \textbackslash \textbackslash
3 10 \textbackslash \textbackslash
3 12 \textbackslash \textbackslash
3 13 \textbackslash \textbackslash
4 6 \textbackslash \textbackslash
4 12 \textbackslash \textbackslash
4 14 \textbackslash \textbackslash
5 6 \textbackslash \textbackslash
5 7 \textbackslash \textbackslash
5 8 \textbackslash \textbackslash
5 9 \textbackslash \textbackslash
6 7 \textbackslash \textbackslash
6 12 \textbackslash \textbackslash
6 14 \textbackslash \textbackslash
7 9 \textbackslash \textbackslash
7 11 \textbackslash \textbackslash
7 13 \textbackslash \textbackslash
7 14 \textbackslash \textbackslash
8 9 \textbackslash \textbackslash
8 10 \textbackslash \textbackslash
8 11 \textbackslash \textbackslash
9 11 \textbackslash \textbackslash
9 13 \textbackslash \textbackslash
9 14 \textbackslash \textbackslash
10 11 \textbackslash \textbackslash
10 12 \textbackslash \textbackslash
10 13 \textbackslash \textbackslash
11 13 \textbackslash \textbackslash
11 14 \textbackslash \textbackslash
12 13 \textbackslash \textbackslash
12 14 \textbackslash \textbackslash
13 14 \textbackslash \textbackslash
\\
Dvo\v{r}\'{a}k output:Always colorable (using standard Alon-Tarsi)\\
}
\end{computation}

\begin{computation}[$G_{19}$]
\label{comp19}
\emph{\phantom{m}\\
Algorithm 1 edges: [(0,1),(0,2),(0,3),(0,4),(1,2),(2,3),(3,4),(1,4),(1,5),(2,5),(3,5),(4,5)]\\
Algorithm 1 listsize: 4\\
Algorithm 1 output: 12 36 \textbackslash \textbackslash
4 4 4 4 4 4 4 4 4 4 4 4 \textbackslash \textbackslash
0 1 \textbackslash \textbackslash
0 2 \textbackslash \textbackslash
0 3 \textbackslash \textbackslash
0 4 \textbackslash \textbackslash
0 5 \textbackslash \textbackslash
0 6 \textbackslash \textbackslash
1 2 \textbackslash \textbackslash
1 3 \textbackslash \textbackslash
1 4 \textbackslash \textbackslash
1 7 \textbackslash \textbackslash
1 8 \textbackslash \textbackslash
2 3 \textbackslash \textbackslash
2 7 \textbackslash \textbackslash
2 9 \textbackslash \textbackslash
2 10 \textbackslash \textbackslash
3 5 \textbackslash \textbackslash
3 9 \textbackslash \textbackslash
3 11 \textbackslash \textbackslash
4 5 \textbackslash \textbackslash
4 6 \textbackslash \textbackslash
4 7 \textbackslash \textbackslash
4 8 \textbackslash \textbackslash
5 6 \textbackslash \textbackslash
5 9 \textbackslash \textbackslash
5 11 \textbackslash \textbackslash
6 8 \textbackslash \textbackslash
6 10 \textbackslash \textbackslash
6 11 \textbackslash \textbackslash
7 8 \textbackslash \textbackslash
7 9 \textbackslash \textbackslash
7 10 \textbackslash \textbackslash
8 10 \textbackslash \textbackslash
8 11 \textbackslash \textbackslash
9 10 \textbackslash \textbackslash
9 11 \textbackslash \textbackslash
10 11 \textbackslash \textbackslash\\
Dvo\v{r}\'{a}k output:Always colorable (using standard Alon-Tarsi)\\
}
\end{computation}


\begin{thebibliography}{00}



\bibitem{AT}
N. Alon and M. Tarsi, Colourings and orientations of graphs, \emph{Combinatorica} \textbf{12} (1992) 125--134.


\bibitem{Bo}
M. Bonamy, Planar graphs with $\Delta \geq 8$ are $(\Delta+1)$-edge-choosable, Seventh Euro. Conference in Comb., Graph Theory and App., CRM series, vol 16. Edizioni della Normale (2013)

\bibitem{B}
O.V. Borodin, A generalization of Kotzig's theorem on prescribed edge coloring of planar graphs, \emph{Mat. Zametki} \textbf{48} (1990), 1186--1190.

\bibitem{BKW}
O.V. Borodin and A. V. Kostochka, and D. R. Woodall, List edge and list total colorings of multigraphs, \emph{J. Combin. Theory Ser. B} \textbf{71} (1997), 184--204. 

\bibitem{bkw2}
O.V. Borodin and A. V. Kostochka, and D. R. Woodall, On kernel-perfect orientations of line graphs. Discrete Math. 191 (1998), 45--49.

\bibitem{BM}
G. Brinkmann and B. McKay, Fast generation of planar graphs, \emph{MATCH Commun. Math. Comput. Chem.} \textbf{58} (2007), 323--357.

\bibitem{CCSS}
D. Cariolaro and G. Cariolaro and U. Schauz and X. Sun, The list-chromatic index of $K_6$, \emph{Discrete Mathematics} \textbf{322} (2014), 15--18.

\bibitem{CH}
N. Cohen and F. Havet, Planar graphs with maximum degree $\Delta \leq 9$ are $(\Delta+1)$-edge-choosable--a short proof, \emph{Discrete Math.} \textbf{310}  (2010), 3049--3051. 

\bibitem{Dv}
Z. Dvo\v{r}\'{a}k, An efficient implementation and a strengthening of Along-Tarsi list coloring method, arxiv:2301.06571 (2023).

\bibitem{EG}
M. Ellingham and L. Goddyn, List edge colourings of some 1-factorable multigraphs, \emph{Combinatorica} \textbf{16} (1996), 343--352.


\bibitem{ERT}
P. Erd\H{o}s and A. Rubin and H. Taylor, Choosability in graphs. \emph{Congr. Numer.} \textbf{26} (1979), 125--157.

\bibitem{G}
F. Galvin, The list chromatic index of a bipartite multigraph, \emph{J. Combin. Theory Ser. B} \textbf{63} (1995), 295--313.


\bibitem{HJ} R. H\"{a}ggkvist and J. Janssen, \emph{Combin. Probab. Comput.} \textbf{6} (1997), 295--313.


\bibitem{jh}
J. Harrelson, List-Edge Coloring Planar Graphs with Bounded Maximum Degree, Ph.D. Thesis, \emph{Auburn University} (2019).

\bibitem{Ho}
I. Holyer, The $NP$-completness of edge-coloring, \emph{SIAM J. Comput.} \textbf{4} (1981), 718--720. 

\bibitem{JT}
T. Jensen and B. Toft, Graph coloring problems, vol. 39, John Wiley \& Sons, 2011.

\bibitem{JMS}
M. Juvan and B. Mohar and R. \v{S}krekovski, List total colorings of graphs, \emph{Combin. Probab. Comput.} \textbf{7} (1998), 181--188.


\bibitem{MP} J. McDonald and G. Puleo, The list chromatic index of simple graphs whose odd cycles intersect in at most one edge, \emph{Discrete Math.} \textbf{341} (2018), 713--72.

\bibitem{Sage} SageMath, The Sage Mathematics Software System (Version 7.4.1), The Sage Developers, 2017, http://www.sagemath.org.

\bibitem{SZ}
D. Sanders and Y. Zhao, Planar graphs of maximum degree seven are class $I$, \emph{J. Combin. Theory Ser. B} \textbf{83} (2001), 201--212.

\bibitem{S}
U. Schauz, Computing the list chromatic index of graphs, \emph{J. Discrete Alg.} \textbf{52-53} (2018), 182--191.

\bibitem{s3}
U. Schauz, Proof of the list edge coloring conjecture for complete graphs of prime degree, \emph{Electron. J. Comb.} \textbf{21} (2014) \#P3.43.


\bibitem{s2}
U. Schauz, Algebraically solvable problems: describing polynomials as equivalent to explicit solutions, \emph{Electron. J. Comb.} \textbf{15} (2008) \#R10.


\bibitem{Vz}
V. Vizing, Colouring the vertices of a graph with prescribed colours, \emph{Diskret. Analiz} \textbf{29} (1976), 3--10. (In Russian)


\bibitem{Viz}
 V. G. Vizing, On an estimate of the chromatic class of a p-graph, \emph{Diskret. Analiz} \textbf{3} (1964), 25–-30. 



\bibitem{Zh}
L. Zhang, Every planar graph with maximum degree 7 is of class 1, \emph{Graphs Combin.} \textbf{16} (2000), 467--495. 

\end{thebibliography}
\end{document}